\documentclass[11pt,reqno]{amsart} 
\usepackage[height=190mm,width=130mm]{geometry}
\usepackage{amsmath}
\usepackage{amssymb,latexsym}
\usepackage{amsrefs}
\usepackage[draft=true]{hyperref}
\usepackage{cite}
\usepackage{a4wide}
\usepackage{amscd}
\usepackage{graphics}
\usepackage{color}
\usepackage[all]{xy}
\usepackage{mathrsfs}
\usepackage{tikz-cd}
\theoremstyle{plain}
\newtheorem{theorem}{Theorem}
\newtheorem{lemma}{Lemma}
\newtheorem{corollary}{Corollary}
\newtheorem{proposition}{Proposition}
\theoremstyle{definition}
\newtheorem{definition}{Definition}
\theoremstyle{remark}

\newtheorem{example}{Example}

\newcommand{\Z}{\ensuremath{\mathbb{Z}}}

\newcommand{\Q}{\ensuremath{\mathbb{Q}}}

\newcommand{\Hom}{\operatorname{Hom}}

\newcommand{\Ext}{\operatorname{Ext}} 

\newcommand{\Ker}{\operatorname{Ker}}
\newcommand{\Rad}{\operatorname{Rad}}

\newcommand{\Soc}{\operatorname{Soc}}

\numberwithin{equation}{section} 

\begin{document}

\title{ On the Rings whose injective right modules are max-projective}

\author{Yusuf Alag\"oz}

\address{S\.{I}\.{I}RT UNIVERSITY\\ DEPARTMENT OF MATHEMATICS\\ S\.{I}\.{I}RT \\ TURKEY}

\email{yusuf.alagoz@siirt.edu.tr}

\author{ENG\.{I}N B\"uy\"uka\c{s}{\i}k}

\address{\.{I}ZM\.{I}R INSTITUTE OF TECHNOLOGY\\ DEPARTMENT OF MATHEMATICS\\ \.{I}ZM\.{I}R\\ TURKEY}

\email{enginbuyukasik@iyte.edu.tr}

\author{Haydar Baran Yurtsever}

\address{\.{I}ZM\.{I}R INSTITUTE OF TECHNOLOGY\\ DEPARTMENT OF MATHEMATICS\\ \.{I}ZM\.{I}R\\ TURKEY}

\email{haydaryurtsever@iyte.edu.tr}

\begin{abstract}
Recently, the rings whose injective right modules are $R$-projective (respectively, max-projective) were investigated and  studied in \cite{maxproj}.  Such ring are called right almost-$QF$ (respectively, max-$QF$).  In this paper, our aim is to give some further characterization of these rings over more general classes of rings, and address several questions about these rings. We obtain characterizations of max-$QF$ rings over several classes of rings including local, semilocal right semihereditary, right nonsingular right noetherian and right nonsingular right finite dimensional rings. We prove that for a ring $R$ being right almost-$QF$ and right max-$QF$ are not left-right symmetric. We also show that right almost-$QF$ and right max-$QF$ rings are not closed under factor rings. This leads to consider the rings all of whose factor rings are almost-$QF$ and max-$QF$.

\end{abstract}



\maketitle

\section{introduction}

A module $M$ is said to be $R$-projective (respectively, max-projective) provided that each homomorphism from $M$ into $R/I$ where $I$ is any right (respectively, maximal) ideal, factors through the canonical projection $\pi:R\rightarrow R/I$. After Faith \cite{FaithQF} asking that ‘when does $R$-projectivity imply projectivity for all right $R$-modules?’, R-projective modules became an important field of research in Ring Theory and Homological Algebra. This problem is also studied in various recent
papers (see, \cite{testingforprojectivity, primehereditarynoether, trifilajfaithproblem, dualbeartrilifaj}). Moreover, $R$-projective (resp. max-projective) modules have been studied from many different perspectives by several authors (see, \cite{maxproj, nonsingular, A-perfect, B-perfect}). For instance, the rings whose flat right $R$-modules are $R$-projective and max-projective are characterized in \cite{A-perfect, B-perfect} and the rings whose nonsingular right $R$-modules are $R$-projective are characterized in \cite{nonsingular}.

\vspace{0.2cm}

A ring $R$ is said to be Quasi-Frobenius ($QF$) if $R$ is right (or left) self injective right (or left) Artinian. $QF$ rings were first introduced in \cite{Nakayama} by Nakayama, in the study of representations of algebras. A result of Faith states that $R$ is $QF$ if and only if every injective right $R$-module is projective (see \cite{FaithQF}). Motivated by this remarkable theorem, Alagöz et al. in \cite{maxproj}, considered the rings whose injective right modules are $R$-projective (respectively, max-projective). They call these rings right almost-$QF$ (respectively, max-$QF$). Several properties and structure of these rings are given in \cite{maxproj} over some particular rings. For example, almost-$QF$ and max-$QF$ rings are completely characterized over commutative Noetherian rings, right Noetherian local rings and right hereditary right Noetherian rings. Although remarkable distance has been recorded in these directions, still there are open problems on the structure of the rings that have been considered.

\vspace{0.2cm}

We first recall the following characterization over right hereditary and right Noetherian rings.

\begin{theorem} \cite[Theorem 1]{maxproj} \label{hereditary-noetherian} Let $R$ be a right Hereditary and right Noetherian ring. The following statements are equivalent.
\begin{enumerate}
\item[(1)] $R$ is right almost-$QF$.
\item[(2)] $R$ is right  max-$QF$.
\item[(3)] Every injective right $R$-module $E$ has a decomposition $E=A\oplus B$ where $\Rad(A)=A$ and $B$ is projective and semisimple.
\item[(4)] $R=S\times T$, where $S$ is a semisimple Artinian and $T$ is a right small ring.
\end{enumerate}
\end{theorem}

The main goal of this paper is to give further ring theoretic characterizations  and to study the structure   of almost-$QF$ and max-$QF$ rings for more general classes of rings.

\vspace{0.2cm}

In this respect, for some classes of rings including, local, semilocal right semihereditary, right nonsingular right finite dimensional and right Noetherian right nonsingular, we obtain some conditions that equivalent to being right max-$QF$.

\vspace{0.2cm}

In section 2, some generalizations of Theorem \ref{hereditary-noetherian} is obtained by replacing right hereditary with right nonsingular, and replacing right noetherian with right finite dimensional. Namely, we prove that a right finite dimensional and right nonsingular ring $R$ is right max-QF if and only if every injective right module $E$ has a decomposition $E= E_{1} \oplus E_{2}$ where $\Rad(E_1)=E_1$ and $E_{2}$ is projective if and only if $E(R_R)$ is max-projective and $\Rad(E)=E$ for every singular injective right $R$-module.

\vspace{0.2cm}

In section 3, we obtain a characterization of right max-$QF$ rings over local rings, and over semilocal right semihereditary rings.
We prove that, a local ring $R$ is right max-QF if and only if $R$ is either right small or right self injective with $\Ext_R(E,J(R))=0$, for each injective right $R$-module $E$. More generally,  we show that a semilocal right semihereditary ring $R$ is right max-QF  if and only if $R = S \times T$, where $S$ is semisimple Artinian and $T$ is right small ring if and only if the dual Goldie torsion theory splits.

\vspace{0.2cm}

Another problem we consider in this paper is, whether for a ring being max-$QF$ and almost-$QF$ is left-right symmetric or not. In section 4, we show that Small's famous Example (see, \cite[2.33]{lam}) is an example of a ring that is right almost-$QF$ (hence, right max-$QF$), but not left max-$QF$ (hence, not left almost-$QF$). Finally, an example is given in order to show that  max-$QF$ and almost-$QF$ rings are not closed under factor rings. This leads to the investigation of right super almost-$QF$ and right super max-$QF$ rings respectively i.e. the rings whose factor rings are right almost-$QF$ and right max-$QF$ respectively.

\section{some classes of max-QF rings}

Ring theoretic characterizations of max-$QF$ rings is only known for some particular classes of rings.
Complete characterization of max-$QF$ rings obtained for right Hereditary right Noetherian rings in \cite[Theorem 1]{maxproj} as it is stated in Theorem \ref{hereditary-noetherian}. This section is devoted to obtain the structure of max-$QF$ rings for more general classes of rings and generalize some results obtained in \cite{maxproj}.  In this context, the first aim is to examining these rings over right nonsingular rings. We start by recalling what is understood by a nonsingularity.

\vspace{0.2cm}

Recall that the singular submodule $Z(M)$ of a right $R$-module $M$ is the set of elements $m\in M$ such that $mI=0$ for some essential right ideal $I$ of $R$. A right module $M$ is called \emph{singular} if $Z(M) = M$, and \emph{nonsingular} if $Z(M)=0$. Thus, $R$ is called a right nonsingular ring if $Z(R_{R})=0$ (see, \cite{goodearl}).

\vspace{0.2cm}

We begin with the following lemma that we use in the sequel.

\begin{lemma}\label{SingularMaxprojective} Let $R$ be a right nonsingular ring and $E$ be a singular injective right $R$-module. Then $E$ is max-projective if and only if $\Rad(E) = E$.
\end{lemma}

\begin{proof}Sufficiency is clear. To prove the necessity, assume that $\Rad(E) \neq E$. Then $E$ contains a maximal submodule, and so there is a nonzero homomorphism $f: E \to R/I$ for some maximal right ideal $I$ of $R$. Since $E$ is singular, and $R$ is nonsingular $\Hom(E,\, R)=0$. Thus the map $f$ can not be lifted to a homomorphism from $E$ to $R$. Hence $E$ is not max-projective. This proves the necessity.
\end{proof}

Recall that a submodule $N$ of a right $R$-module $M$ is said to be closed in $M$, if $N$ has no proper essential extension in $M$. It is well known that every closed submodule $N$ of an injective module $M$ is a direct summand in $M$ (see, \cite{lam}).

\begin{lemma}\label{indecomposablenonsingularinjective} Let $R$ be a right nonsingular ring and $E$ be a nonsingular indecomposable injective right $R$-module. Then $E$ is max-projective if and only  if either $\Rad (E)=E$ or $E$ is projective and cyclic.

\end{lemma}

\begin{proof}

Sufficiency is clear. Suppose $\Rad(E)\neq E$ and let us show that $E$ is projective. Since $\Rad(E) \neq E$, then there exists a maximal submodule $K$ of $E$ and the corresponding simple factor module $E/K$. Since $E/K$ is simple, then $E/K \cong R/I$ for some maximal right ideal $I$ of $R$. Thus there exists a nonzero homomorphism $f: E \to R/I$. Now, consider the following diagram.

\begin{equation*}
\begin{tikzcd}
  \quad      & E \arrow{d}{f} &   \\
R  \arrow{r}{\pi} & R/I   \arrow{r}  &  0
\end{tikzcd}
\end{equation*}

Since $E$ is max-projective by assumption, there is  a nonzero homomorphism $g:E \to R$ such that $f=\pi g$. By the First Isomorphism Theorem, $E/\Ker(g) \cong Im(g) \subseteq R$, and by the fact that $R_R$ is nonsingular, $E/\Ker(g)$ is also nonsingular. Thus, $\Ker(g)$ is closed in $E$ by \cite[Lemma 2.3]{sandomierski-closed}. As the closed submodules of injective modules are direct summands, and $\Ker(g)$ is a closed submodule of $E$, we have $E \cong \Ker(g) \oplus E'$ for some submodule $E'$ of $E$. Now, since $E$ is indecomposable,  $\Ker(g)=0$ or $E'=0$. If $E'=0$, then $\Ker(g)=E$ so that $g=0$, a contradiction. Thus $\Ker(g)=0$ and $E=E'$, means that $g$ is a monomorphism. Since $g$ is a monomorphism, $g(E) \cong E$ is injective.  So that $R = g(E) \oplus J$ for some right ideal $J$ of $R$. Now, since $R$ is projective and $g(E)$ is direct summand of $R$,  $g(E) \cong E$ is cyclic and projective. This proves the necessity.
\end{proof}

In \cite{matlis}, Matlis proved that a ring $R$ is right Noetherian if and only if every injective right module $M$ can be written as a direct sum of indecomposable (injective) submodules. Using this crucial result of Matlis, when we replace the right hereditary assumption from the Theorem \ref{hereditary-noetherian} with the right nonsingularity, we have the following results that characterizes the $R$-projective modules and the almost-$QF$ rings.

\begin{lemma}\label{noethnonsingularinjectives}

  Let R be a right Noetherian and right nonsingular ring. Then the following statements are equivalent for an injective right module $E$.

  \begin{enumerate}

    \item[(1)] E is $R$-projective

    \item[(2)] $E$ is max-projective

    \item[(3)] $E= E_{1} \oplus E_{2}$ where $\Rad(E_1)=E_1$ and $E_2$ is projective.

  \end{enumerate}

\end{lemma}

\begin{proof}

  $(1) \Rightarrow (2)$ is clear.

 $(2) \Rightarrow (3)$ Since $R$ is right Noetherian, every injective right $R$-module is a direct sum of indecomposable injective right $R$-modules. Thus $E= { \oplus}_{ i \in I} E_i $ where $I$ is an index set and $E_i$ is indecomposable for each $i \in I$. As $E$ is max-projective and $R$ is right nonsingular, $E_i$ is projective or $\Rad(E_i)= E_i$ for each $i \in I$ by Proposition \ref{indecomposablenonsingularinjective}. Let $ J= \{ i \in I \mid \Rad(E_i)= E_i \}$.  Then for $E_1= \oplus _{i \in J}E_i$ and $E_2=\oplus _{ i\in I \backslash J}E_i$, we have $\Rad(E_1)= E_1$ and $E_2$ is projective. Thus (3) follows.

  $(3) \Rightarrow (1)$ Let $Q$ be an injective right R-module. Then, by (3), $Q = Q_1 \oplus Q_2$ where $\Rad(Q_1)= Q_1$ and $Q_2$ is projective. Since $R$ is right Noetherian, $R/I$ is Noetherian for each right ideal $I$ of $R$. Thus, $\Hom( Q_1, R/I )=0$ for each right ideal $I$ of $R$. So that $Q_1$ is $R$-projective. On the other hand, $Q_2$ is also $R$-projective by projectivity of itself. Hence $Q = Q_1 \oplus Q_2$ is $R$-projective as a direct sum of $R$-projective modules.
\end{proof}

\begin{proposition} \label{noethnonsingularMaxQF}

  Let $R$ be a right Noetherian and right nonsingular ring. Then the following statements are equivalent.

  \begin{enumerate}

    \item[(1)] $R$ is right almost-$QF$.

    \item[(2)] $R$ is right max-$QF$.

    \item[(3)] Every injective right module $E$ has a decomposition $E= E_{1} \oplus E_{2}$ where $\Rad(E_1)=E_1$ and $E_{2}$ is projective.

  \end{enumerate}

\end{proposition}

\begin{proof}$(1) \Rightarrow (2)$ is clear. $(2) \Rightarrow (3)$ By Lemma \ref{noethnonsingularinjectives}.

$(3) \Rightarrow (1)$ Every projective module is $R$-projective,  and each module $N$ with $\Rad(N)=N$ is $R$-projective over a right Noetherian ring. Hence $(3)$ implies $(1)$.
\end{proof}

A right $R$-module $M$ is said to be \emph{finite dimensional} provided that $M$ contains no infinite independent families of nonzero submodules. For example, all Noetherian modules are finite dimensional. A ring $R$ is said to be \emph{finite dimensional} if the right $R$-module $R_{R}$ is finite dimensional (\cite{goodearl}). Now, from the Theorem \ref{hereditary-noetherian}, if we replace the right hereditary right Noetherian assumption on the ring with the right finite-dimensional right nonsingular, we have the following crucial characterization of max-$QF$ rings.

\begin{theorem}\label{prop:findimMaxQF}

Let $R$  be a right finite-dimensional and right nonsingular ring. Then the following are equivalent.

\begin{enumerate}

  \item[(1)] $R$ is right max-$QF$.

  \item[(2)] $E(R_R)$ is max-projective and $\Rad(E)=E$ for every singular injective right $R$-module $E$.

  \item[(3)] Every injective right $R$-module $E$ has a decomposition $E= E_{1} \oplus E_{2}$ where $\Rad(E_1)=E_1$ and $E_{2}$ is projective.

  \item[(4)] Every nonsingular injective right $R$-module is max-projective and $Rad(E)=E$ for every singular injective right $R$-module $E$.

\end{enumerate}

\end{theorem}

\begin{proof} $(1) \Rightarrow (2)$ and $(3) \Rightarrow (1)$ are clear.

$(2) \Rightarrow (3)$  Let $E$ be an injective right $R$-module. Since $R$ is right nonsingular, $E=Z(E)\oplus K$ for some submodule $K$ of $E$. Note that $K$ is right nonsingular injective right $R$-module, and $\Rad(Z(E))=Z(E)$ by Lemma  \ref{SingularMaxprojective}. Then $K$ is a direct sum of indecomposable injective modules by \cite[Example 3B-5]{goodearl}, that is $K=\oplus_{i \in I}K_i$, where each $K_i$ is indecomposable and injective. Then for each $i \in I$, $K_i$ is projective or $\Rad(K_i)=K_i$ by Lemma \ref{indecomposablenonsingularinjective}. Then $K$ can be expressed as $K=K_1 \oplus K_2$, where $\Rad(K_1)=K_1$ and $K_2$ is projective. For $E_1=Z(E)\oplus K_1$, and $E_2=K_2$, $E$ has the desired decomposition in $(3)$.

$(1) \Rightarrow (4)$ Clearly $(1)$ implies that nonsingular injective right $R$-modules are max-projective. By Lemma \ref{SingularMaxprojective}, we have $\Rad(E)=E$ for every singular right $R$-module.

$(4) \Rightarrow (1)$ Since $R$ is right nonsingular, every injective right module $Q$ can be written as $Q=Z\oplus N$, where $Z$ is the singular submodule of $Q$ and $N$ is nonsingular. By $(4)$, $\Rad (Z)=Z$, hence it is max-projective. As $N$ is nonsingular, $N$ is max-projective again by $(4)$. Hence $Q$ is max-projective, and so $R$ is right max-QF.
\end{proof}

\begin{lemma}\label{lem:indnonembedsinE(R)} Let $R$ be a right nonsingular ring and $Q$ be an indecomposable nonsingular injective right $R$-module. Then $Q$ embeds in $E(R_R)$.

\end{lemma}

\begin{proof} Let $0 \neq x \in Q$. Then $xR \cong R/I$ for some closed right ideal $I$ of $R$. Let $J$ be a complement of $I$ in $R$. Then $J \cong (J \oplus I)/I$ is essential in $R/I$. Thus $xR$ contains an essential submodule isomorphic to $J$, say $K$.  Since $Q$ is indecomposable and injective, it is uniform, whence $K$ is essential in $Q$. Therefore, $Q=E(K)\cong E(J) \subseteq E(R_R)$ which means that $Q$ embeds in $E(R_R)$. This completes the proof.
\end{proof}

Let $R$ be any ring and $M$ be an $R$-module. A submodule $N$ of $M$ is called radical submodule  if $N$ has no maximal submodules, i.e. $N = \Rad(N)$. By $P(M)$ we denote the sum of all radical submodules of a module $M$. Then $P(M)$ is the largest radical submodule of $M$, and so $\Rad(P(M))=P(M)$. Moreover, $P$ is an idempotent radical with $P(M)\subseteq \Rad(M)$ and $P(M/P(M))=0$, (see \cite{radsupp}).

\begin{lemma}\label{lem:P(E)=0} If $R$ is a right nonsingular and right max-QF ring with $P(E(R_R))=0$, then every indecomposable nonsingular injective right $R$-module is projective.

\end{lemma}

\begin{proof} Let $K$ be an indecomposable nonsingular injective right $R$-module. Then $K$ is max-projective by the hypothesis, and so $K$ is projective or $\Rad(K)=K$ by Lemma \ref{indecomposablenonsingularinjective}. On the other hand, $K$ embeds in $E(R_R)$ by Lemma \ref{lem:indnonembedsinE(R)}. As the hypothesis says that $P(E(R_R))=0$, $\Rad(K)=K$ is not possible. Therefore $K$ is projective.
\end{proof}

We obtain the following corollary by Theorem \ref{prop:findimMaxQF} and Lemma \ref{lem:P(E)=0}.

\begin{corollary} Let $R$ be a right finite dimensional and right nonsingular ring with $P(E(R_R))=0$. The following  are equivalent.

\begin{enumerate}

  \item[(1)] $R$ is right max-$QF$.

  \item[(2)] $E(R_R)$ is projective and $\Rad(E)=E$ for every singular injective right $R$-module $E$.

  \item[(3)] Every injective right $R$-module $E$ has a decomposition $E= E_{1} \oplus E_{2}$ where $\Rad(E_1)=E_1$ and $E_{2}$ is projective.

  \item[(4)] Every nonsingular injective right $R$-module is projective and $\Rad(E)=E$ for every singular injective right $R$-module $E$.

\end{enumerate}

\end{corollary}

\section{local and semilocal max-$QF$ rings}

In this section, we give some characterizations of max-$QF$ rings over local rings and semilocal right semihereditary rings.
In particular, we will characterize semilocal right semihereditary  rings whose dual Goldie torsion theory is splitting. V.S. Ramamurthi defined a dual Goldie torsion theory and studied some of its properties in \cite{ramamurthi}. Özcan and Harmancı proved that over a $QF$ ring the dual Goldie torsion theory splits. They raised the question, "Is a ring $R$ whose dual Goldie torsion theory is splitting a $QF$ ring?" (see, \cite{ozcan}). Later, in \cite{lomp2}, it is given by Lomp a list of classes of rings whose dual Goldie torsion theory is splitting, but that are far from being $QF$. Now, over a semilocal right semihereditary ring, we give a complete characterization of right max-$QF$ rings in terms of splitting of the dual Goldie torsion theory. This characterization will provide another negative example to question mentioned above.
We recall the following  lemma, which states that finite direct product of right max-$QF$ rings is also right max-$QF$.
\begin{lemma} \cite[Lemma 4]{maxproj} \label{product}
Let $R_{1}$ and $R_{2}$ be rings. Then $R=R_{1}\times R_{2}$ is right max-$QF$ if and only if $R_{1}$ and $R_{2}$ are both right max-$QF$.
\end{lemma}

Recall that $R$ is said to be right semihereditary if all finitely generated right ideals are projective. A result of Megibben, see \cite[Theorem 2]{megibben}, states that a ring $R$ is right semihereditary if and only if every quotient module of an $FP$-injective right module is $FP$-injective. Now we are ready to characterize right max-$QF$ rings whose dual Goldie torsion theory is splitting.

\begin{theorem}\label{prop:semilocalrightsemihereditary} Let $R$ be a semilocal right semihereditary ring. Then the following statements are equivalent.

\begin{enumerate}

  \item[(1)] $R$ is right max-$QF$.

  \item[(2)] $R = S \times T$, where $S$ is semisimple Artinian and $T$ is right small ring.

  \item[(3)] Every simple injective right module is projective.

  \item[(4)] Every singular injective right module is max-projective.

  \item[(5)] The dual Goldie torsion theory splits.

\end{enumerate}

\end{theorem}

\begin{proof} $(1) \Rightarrow (2)$ Let $S$ be the sum of the injective minimal right ideals of $R$. So $S$ is an ideal of $R$ and $S \cap J(R) = 0$.
 Thus $S$ embeds in $R/J(R)$, and so $S$ is finitely generated and injective. Then we can decompose $R$ as $R = S \times T$. Since $R$ is right max-QF, $T$ is right max-QF as well by Lemma \ref{product}. Suppose $T$ is not right small. Let $K$ be a maximal submodule of $E = E(T_T)$. Then $E/K$ is FP-injective by \cite[Theorem 2]{megibben}. Since $R$ is semilocal, $E/K$ is also pure injective by \cite[Proposition 4.3]{simples-pure-inj}. Being FP-injective and pure injective implies that $E/K$ is an injective $T$-module. Then $E/K$ is also injective as a right $R$-module by \cite[Example 3.11A]{lam}. As $T$ is right max-QF, $E/K$ is a max-projective $T$-module. Thus, $T = X \oplus Y$ for some right ideals $X$, $Y$ such that $X \cong E/K$.  So we obtain that, $X$ is a simple injective right R-module and $S \cap X = 0$, a contradiction. Therefore $T$ is a right small ring and this proves $(2)$.

$(2) \Rightarrow (1)$ Since every semisimple Artinian ring and every right small ring is max-$QF$, (1) follows by Lemma \ref{product}.

$(2) \Rightarrow (3) \Leftrightarrow (5)$ are clear by \cite[Theorem 4.6]{lomp2}.

$(1) \Rightarrow (4) \Rightarrow (3)$ are clear by \cite[Theorem 2]{maxproj}.

Now, we only need to prove $(3) \Rightarrow (1)$: Let $E$ be an injective right $R$-module and $f: E \to S$ be a homomorphism with $S$ simple. If $f=0$ then there is nothing to prove, so assume that $f \neq 0$. In this case, $f$ is an epimorphism since $S$ is simple. Since $R$ is right semihereditary, $f(E) \cong S$ is $FP$-injective by \cite[Theorem 2]{megibben}. On the other hand, since $R$ is semilocal, $S$ is pure-injective by \cite[Proposition 4.3]{simples-pure-inj}. Thus, $S$ is injective and so is projective by (3). Hence the natural map $\pi: R\to S$ splits, i.e., there exists $ {\pi}' : S \to R$ such that $\pi{\pi}' = 1_S$. Then, $\pi{\pi}'f = f$, and so $E$ is max-projective.
\end{proof}

We obtain the following corollary from Theorem \ref{prop:semilocalrightsemihereditary}.

\begin{corollary} Let $R$ be an indecomposable semilocal right semihereditary ring. Then $R$ is right max-$QF$ if and only if $R$ is semisimple artinian or right small.\\
In particular, if $R$ is a local ring, then $R$ is right max-$QF$ if and only if $R$ is right small or a division ring.
\end{corollary}

In \cite[Proposition 14]{maxproj}, it was proved that if $R$ is a local right max-$QF$ ring, then $R$ is either right self injective or right small. Moreover, in \cite[Corollary 9]{maxproj} it was shown that over a local right noetherian ring, $R$ is right max-$QF$ if and only if $R$ is $QF$ or right small. Motivated by the aforementioned results, we have the following corresponding characterization for max-$QF$ rings over local rings.

\begin{proposition} Let $R$ be a local ring. The following are equivalent.

 \begin{enumerate}

   \item[(1)] $R$ is right max-$QF$.

   \item[(2)] \begin{enumerate}

 \item[(a)] $R$ is right small; or

 \item[(b)] $R$ is right self injective and $\Ext_R(E,J(R))=0$, for each injective right $R$-module $E$.
 \end{enumerate}
\end{enumerate}

\end{proposition}

\begin{proof} $(1) \Rightarrow (2)$ By \cite[Proposition 14]{maxproj}, $R$ is right small or right self-injective. Suppose $R$ is right self injective, and $E$ an injective right R-module. Consider the short exact sequence $0\rightarrow J\rightarrow R\rightarrow R/J\rightarrow 0$ where $J=J(R)$, the Jacobson radical of $R$. Applying ${\rm Hom}_R(E,-)$, we obtain the sequence:
 $${\Hom}_{R}(E,R) \rightarrow {\Hom}_{R}(E,R/J) \rightarrow {\Ext}_{R}^{1}(E,J) \rightarrow {\Ext}_{R}^{1}(E,R)$$


%

Since $R$ is right self-injective, ${\Ext}_{R}^{1}(E,R)=0$. On the other hand, since $R$ is right max-QF, the map ${\Hom}_{R}(E,R) \to {\Hom}_{R}(E,R/J)$ is onto. Therefore ${\Ext}_{R}^{1}(E,J)=0$. This proves (2).

$(2) \Rightarrow (1)$ Suppose $(a)$, that is $R$ is right small. Then $\Rad(E)=E$ for each injective right $R$-module. Since $R/J(R)$ is simple, $\Rad(R/J(R))=0$. As $f(\Rad(M)) \subseteq \Rad(N)$ for each right modules $M$, $N$ and $f \in {\Hom}_{R}(M,N)$, we have that ${\Hom}_{R}(E,R/J(R))=0$. Therefore, $E$ is trivially max-projective and so $R$ is right max-QF.

Now, assume $(b)$. Then for each injective right R-module $E$, the short exact sequence $0\rightarrow J\rightarrow R\rightarrow R/J\rightarrow 0$ where $J=J(R)$, the Jacobson radical of $R$, induces the sequence: ${\Hom}_{R}(E,R) \rightarrow {\Hom}_{R}(E,R/J) \rightarrow {\Ext}_{R}^{1}(E,J)$. By $(b)$ we have ${\Ext}_{R}^{1}(E,J)=0$. Thus ${\Hom}_{R}(E,R) \rightarrow {\Hom}_{R}(E,R/J)$ is onto, and so $E$ is max-projective. Hence $R$ is right max-$QF$.
\end{proof}

%

Recall that $R$ is a right PF ring if $R_{R}$ is an injective cogenerator for the category of right $R$-modules, equivalently $R$ is right self-injective and right socle of $R$ is finitely generated and essential in $R_{R}$ (see, \cite[Theorem 1.56]{quasi-frobenious}). In the following result we give some necessary conditions for a ring to be right max-$QF$.

\begin{proposition} Let $R$ be a local  nonsmall right max-QF ring. Then $R$ is a right self-injective ring and $R$ satisfies one of the following conditions:

\begin{enumerate}

\item[(1)] $R$ is a right PF ring, and every injective right module $E$ has a decomposition $E=E_1 \oplus E_2$, where $\Soc(E_1)$ is essential in $E_1$, $\Rad(E_2)=E_2$ and $\Soc(E_2)=0$.

\item[(2)] Every injective right $R$-module $E$ has a decomposition $E=E_1 \oplus E_2$, where $\Soc(E_1)$ is essential in $E_1$, $\Rad(E_1)=E_1$, and $\Soc(E_2)=0$.
\end{enumerate}
\end{proposition}

\begin{proof} Since $R$ is right max-$QF$ and nonsmall, $R$ is right self injective by \cite[Proposition 14]{maxproj}. Now we divide the proof into two cases:

$\mathbf{Case \ I:}$ Assume $\Soc(R_R) \neq 0$. As $R$ is local, it is indecomposable. Thus $R$ is right uniform, as it is right self-injective. Therefore $\Soc(R_R)$ is simple, and so $R$ is a right PF ring by \cite[Theorem 1.56]{quasi-frobenious}. Let $E$ be an injective right $R$-module and $E_1 = E(\Soc(E))$. Then $E = E_1 \oplus E_2$ with $\Soc(E_2)=0$. Let us show that $\Rad(E_2)=E_2$. Assume that $\Rad(E_2) \neq E_2$. Let $f : E_2 \to R/J(R)$ be a nonzero homomorphism. Since $E_2$ is max-projective, there is a homomorphism $g:E_2 \to R$ such that $f = \pi \circ g$, where $\pi : R \to R/J(R)$ is the natural epimorphism. As $f$ is injective and $\pi$ is a small epimorphism, $g$ is an epimorphism. Then $g$ splits because $R$ is projective. So $E_2 \cong R \oplus \Ker(g)$. This isomorphism and $\Soc (R_R) \neq 0$ implies that $\Soc(E_2) \neq 0$, a contradiction. Thus, $\Rad(E_2)=E_2$, and so $E$ has the desired decomposition. This proves (1).

$\mathbf{Case \ II:}$ Now, assume that $\Soc(R_R)=0$. As in the first case, for any injective right module $E$, we have $E= E_1 \oplus E_2$, where $E_1 = E(\Soc(E))$. Clearly $E_1$ has essential socle, because every module is essential in its injective hull. Also $\Soc(E_2) = E_2$. Let us show that $\Rad(E_1) = E_1$. Assume that $\Rad(E_1) \neq E_1$. Let $f: E_1 \to R/J(R)$ be a nonzero homomorphism. By similar arguments as in the Case I, we have $E_1 \cong R \oplus \Ker(g)$ for some homomorphism $g: E_1 \to R$.
This isomorphism implies that $E_1$ has a nonzero submodule, say $X$, isomorphic to $R$. Since $\Soc(R)=0$, we also have $\Soc(X)=0$. Thus $X \cap \Soc(E_1)=0$. This contradicts with the fact that $\Soc(E_1)$ is essential in $E_1$. Thus, we must have $\Rad(E_1)=E_1$. This proves (2).
\end{proof}

\begin{lemma} \label{hereditary}
Let $R$ be a ring and $M$ a right R-module. If $M/\Rad(M)$ is max-projective, then $M$ is max-projective.
\end{lemma}

\begin{proof}
  Let $f:M/\Rad(M) \to R/I$ be homomorphism, where $I$ is a maximal right ideal of $R$. Then $f \circ \eta :M \to R/I$ is a homomorphism, where $\eta : M \to M/\Rad(M)$ is the natural epimorphism. Since $\Rad(M)\subseteq \Ker(f)$, by the First Isomorphism Theorem, there exists a homomorphism $\bar{f}: M/\Rad(M) \to R/I$ such that $\bar{f} \circ \eta = f$. Since $M/\Rad(M)$ is max projective, there exist a homomorphism $g: M \to R$ such that
  $\pi \circ g = \bar{f} \circ \eta$ where $\pi : R \to R/I$ is the natural epimorphism. Now, compose $\pi \circ g = \bar{f}$ with $\eta$ from the right we got $\pi \circ (g \circ \eta) = \bar{f} \circ \eta = f$.\\
  Hence $g \circ \eta$ lifts $f$, i.e. $M$ is max-projective.
\end{proof}

Recall that $R$ is said to be right hereditary if all right ideals are projective, equivalently every quotient of injective right modules is injective. In \cite[Theorem 2]{maxproj}, over a right hereditary ring, some characterizations of max-$QF$ rings are given. Using Lemma \ref{hereditary}, the following proposition will give further characterizations of max-$QF$ rings.

\begin{proposition} The following are equivalent for a right hereditary ring $R$.
\begin{enumerate}

  \item[(1)] $R$ is right max-$QF$.

  \item[(2)] For every injective right module $E$, the module $E/\Rad(E)$ is max-projective.

  \item[(3)] Every injective right module $E$ with $\Rad(E)=0$ is max-projective.

\end{enumerate}

\end{proposition}

\begin{proof} $(1) \Rightarrow (2)$ straightforward by the right hereditary assumption.

$(2) \Rightarrow (3)$ is clear.

$(3) \Rightarrow (1)$ Let $E$ be an injective right $R$-module. Since $E/\Rad(E)$ has zero radical, $E/\Rad(E)$ is max-projective by $(3)$. Then $E$ is max-projective by Lemma \ref{hereditary}. Thus, $R$ is right max-$QF$.
\end{proof}

\newpage

\section{max-QF and almost-QF rings are not left-right symmetric}

In this section, we show that, for a ring, being max-$QF$ and almost-$QF$ are not left-right symmetric. Small's famous example, is an example of such a ring. 

\begin{proposition} [\bf{Small's Example}]
Consider the ring $R=\begin{pmatrix}
  \Z & \Q\\
  0 & \Q
\end{pmatrix}$.\\
Then $R$ has the following properties.
\begin{enumerate}
  \item[(1)] R is right Noetherian but not left Noetherian.
  \item[(2)] R is right hereditary left semihereditary but not left hereditary.
\end{enumerate}
\end{proposition}

\begin{proof}
 The proof of (1) follows from \cite[1.22]{lamfc} and of (2) follows from \cite[2.33]{lam}.
\end{proof}

\begin{lemma}\label{lem:map} The map $\phi: \begin{pmatrix}

  \Z & \Q\\

  0 & \Q

\end{pmatrix} \to \Q$  given by $\phi (\begin{pmatrix}

  a & b\\

  0 & c

\end{pmatrix})=c$ is a ring epimorphism, and $\Ker(\phi)=\begin{pmatrix}

  \Z & \Q\\

  0 & 0 \end{pmatrix}$ is a maximal ideal of $R$.
\end{lemma}

\begin{proof} It is easy to check that $\phi$ is a ring epimorphism. The first isomorphism theorem and the fact that $\Q$ is a field, implies that $Ker(\phi)$ is a maximal ideal of $R$.

\end{proof}

\begin{lemma}\label{lem:Sisinjective} With the notations in Lemma \ref{lem:map}, the simple left $R$-module $S=R/\Ker(\phi)$ is  singular and injective.

\end{lemma}

\begin{proof} Set $L=\Ker(\phi)$. Since $L$ is a maximal left ideal, it is essential or a direct summand of $R$. Assume that $L\oplus K=R$ for some left ideal $K$ of $R$. Then $l+k=\begin{pmatrix}

  1 & 0\\

  0 & 1

\end{pmatrix}$ for some $l \in L$ and $k \in K$. Thus $K$ contains an element of the form $\begin{pmatrix}

  a & b\\

  0 & 1

\end{pmatrix}$, where $a \in \Z$ and $b \in \Q$. Then $$\begin{pmatrix}

  0 & 1\\

  0 & 0

\end{pmatrix}\cdot \begin{pmatrix}

  a & b\\

  0 & 1

\end{pmatrix}=\begin{pmatrix}

  0 & 1\\

  0 & 0

\end{pmatrix} \in K\cap L$$ a contradiction. Therefore $L$ is an essential left ideal of $R$, and so $S =R/L$ is singular by \cite[Proposition 1.20(b)]{goodearl}.

















Now let us prove that $S$ is an injective left $R$-module. For this purpose, we shall use the Baer's Criteria for injectivity. First note that by \cite[2.33]{lam}, the left ideals of $R$ are  of the following form:
\begin{align*}
  N_{1} &= {\begin{pmatrix}
  n\Z & \Q\\
  0 & \Q
\end{pmatrix}} \qquad (n \neq 0) \\
  N_{2} &= {\begin{pmatrix}
  0 & \Q\\
  0 & \Q
\end{pmatrix}} \\
  N_{V} &= \{{\begin{pmatrix}
  x & y\\
  0 & 0
\end{pmatrix}} : (x,y) \in V \quad \text{( a subgroup of $\Z \oplus \Q)$}\}
\end{align*}

We shall show that,  each nonzero homomorphism $f$ from $N_1,\,N_2$ and $N_V$ to $S$ can be extended to a homomorphism $g: R \to S$.

{\bf I:} Suppose $f: N_{1} \to S$ is a nonzero homomorphism. Since $S$ is a  simple left ideal and $f$ is nonzero,  $Im(f)=S$. Then  $\Ker(f)$ is a maximal submodule of $N_{1}$.  It is easy to check that, maximal submodules of $N_{1}$ are of the form:\\
\begin{align*}
  K_{p} &= {\begin{pmatrix}
  pn\Z & \Q\\
  0 & \Q
\end{pmatrix}} \qquad (p\,\, \text{is prime}) \\
  K &= {\begin{pmatrix}
  n\Z & \Q\\
  0 & 0
  \end{pmatrix}}
\end{align*}
If $\Ker(f)= K_{p}$, then $N_1 /K_p \cong S$. But

$$ ann_l(N_1 /K_p)= {\begin{pmatrix}
  p\Z & \Q\\
  0 & \Q
\end{pmatrix}} \neq {\begin{pmatrix}
  \Z & \Q\\
  0 & 0
\end{pmatrix}} =ann_l (S).$$

This is a contradiction, because isomorphic modules have the same annihilator ideal. Therefore $\Ker(f)\neq K_p$. \\

If $\Ker(f)=K$, then for ${\begin{pmatrix}
  nk & a\\
  0 & b
\end{pmatrix}} \in N_1$ we have

$$f{\begin{pmatrix}
  nk & a\\
  0 & b
\end{pmatrix}}=f{\begin{pmatrix}
  nk & a\\
  0 & 0
\end{pmatrix}}+ f{\begin{pmatrix}
   0& 0\\
  0 & b
\end{pmatrix}}=f{\begin{pmatrix}

  0 & 0\\

  0 & b

\end{pmatrix}}={\begin{pmatrix}

  0 & 0\\

  0 & b

\end{pmatrix}} \cdot f{\begin{pmatrix}

  0 & 0\\

  0 & 1

\end{pmatrix}}$$

Let $f{\begin{pmatrix}

  0 & 0\\

  0 & b

\end{pmatrix}}= {\begin{pmatrix}

  0 & 0\\

  0 & s

\end{pmatrix}} + {\begin{pmatrix}

  \Z & \Q \\

  0 & 0

\end{pmatrix}}.$ Then, it is easy to check that

$f{\begin{pmatrix}

  nk & a\\

  0 & b

\end{pmatrix}}={\begin{pmatrix}

  nk & a\\

  0 & b

\end{pmatrix}}\cdot f{\begin{pmatrix}

  0 & 0\\

  0 & 1

\end{pmatrix}}. $ That is $f$ is the right multiplication by

$f{\begin{pmatrix}

  0 & 0\\

  0 & 1

\end{pmatrix}}$. Thus the map $g: R \to S$ defined by $g{\begin{pmatrix}

  1 & 0\\

  0 & 1

\end{pmatrix}}=f{\begin{pmatrix}

  0 & 0\\

  0 & 1

\end{pmatrix}}$ extends $f$.

In conclusion , we see that each homomorphism $f: N_1 \to S$ extends to a homomorphism $g: R \to S$.

{\bf II:} Consider the left ideal $N_{2}={\begin{pmatrix}

  0 & \Q\\

  0 & \Q

\end{pmatrix}}$ and let $f:N_{2}\to S$ be a homomorphism. Since

\begin{equation*}
  {_RR=\begin{pmatrix}

  \Z & \Q\\

  0 & \Q

\end{pmatrix}} = {\begin{pmatrix}

  \Z & 0\\

  0 & 0

\end{pmatrix}} \oplus {\begin{pmatrix}

  0 & \Q\\

  0 & \Q

\end{pmatrix},}
\end{equation*} $N_{2}$ is a direct summand of \textit{R}. Let $\pi:R\to N_{2}$ be the projection homomorphism and  $i: N_2 \to R$ the inclusion homomorphism. Then the homomorphism $g = f \circ \pi. :R \to S$ extends $f$.

{\bf III:} Consider the left ideal $N_{V}$ = $\{{\begin{pmatrix}

  x & y\\

  0 & 0

\end{pmatrix}}$ : $(x,y) \in V$ \text{(V is a subgroup of $\Z \oplus \Q$)}\}. Since ${\begin{pmatrix}

  n & q_1\\

  0 & q_2

\end{pmatrix}}\cdot {\begin{pmatrix}

  x & y\\

  0 & 0

\end{pmatrix}}=  {\begin{pmatrix}

  nx & ny\\

  0 & 0

\end{pmatrix}} $, the left multiplication on $N_V$ is determined by $\Z$. Therefore the lattice of left submodules of $N_V$ and that of $V$ are isomorphic. Therefore, for each maximal submodule $K$ of $N_V$,  $N_V/K\cong \Z_p$ for some prime integer $p$. Hence, as $S$ is infinite,  $\Hom(N_V,\, S)=0$. This means that, any homomorphism from $N_V \to S$ extends trivially to a homomorphism $R \to S.$

Thus, by Baer Criteria, $S$ is an injective simple left $R$-module. This completes the proof.
\end{proof}

\begin{lemma}

  R=${\begin{pmatrix}

  \Z &  \Q\\

  0 & \Q

\end{pmatrix}}$ is not left small in the left $R$-module $W$=${\begin{pmatrix}

  \Q & \Q\\

  \Q & \Q

\end{pmatrix}}$

\end{lemma}

\begin{proof}

 Consider $X$ = ${\begin{pmatrix}

  0 & \Q\\

  0 & \Q

\end{pmatrix}}$. $X$ is a proper left submodule of $W$ and $R + X = W$. Thus, $R$ is not left small in $W$.
\end{proof}

\begin{lemma}\label{Rightsmall}

   $R={\begin{pmatrix}

  \Z &  \Q\\

  0 & \Q

\end{pmatrix}}$ is right small in the right $R$-module $T$=${\begin{pmatrix}

  \Q & \Q \\

  \Q & \Q

\end{pmatrix}}.$

\end{lemma}

\begin{proof} As a  right $R$-module, submodules of $T$ are of the following form:

\begin{equation*}
 N_{U} = \{{\begin{pmatrix}

  x & \Q\\

  y & \Q

\end{pmatrix}} : (x,y) \in U \text{(U is a subgroup of $\Q \oplus \Q$)\}},
\end{equation*}

\begin{equation*}
 {\begin{pmatrix}
  \Q  &  \Q\\
  0 & 0
\end{pmatrix}},
{\begin{pmatrix}
  0 &  0\\
  \Q & \Q
\end{pmatrix}},
{\begin{pmatrix}
  0 & \Q \\
  0 &  0
\end{pmatrix}},
{\begin{pmatrix}
  0 &  0\\
  0 & \Q
\end{pmatrix}}.
\end{equation*}
We see that $R +X\neq T$ for each proper submodule $X$ of $T$. Hence $R$ is a right small submodule of $T$.
\end{proof}

\begin{lemma}\label{rightRessential}

  Let $R$=${\begin{pmatrix}

  \Z &  \Q\\

  0 & \Q

\end{pmatrix}}$. Then $R_R$ is  essential in $W_{R}$=${\begin{pmatrix}

  \Q & \Q\\

  \Q & \Q

\end{pmatrix}}.$

\end{lemma}

\begin{proof}  Let ${\begin{pmatrix}

  a & b\\

  c & d

\end{pmatrix}}$ be a nonzero element of $W$. \\

Case I: If $a\neq 0$ or $c \neq 0$, then ${\begin{pmatrix}

  a & b\\

  c & d

\end{pmatrix}} \cdot {\begin{pmatrix}

  0 & 1\\

  0 & 0

\end{pmatrix}} ={\begin{pmatrix}

  0 & a\\

  0 & c

\end{pmatrix}}$ is a nonzero element of $R$.\\

Case II: If $b\neq 0$ or $d \neq 0$, then ${\begin{pmatrix}

  a & b\\

  c & d

\end{pmatrix}} \cdot {\begin{pmatrix}

  0 & 0\\

  0 & 1

\end{pmatrix}} ={\begin{pmatrix}

  0 & b\\

  0 & d

\end{pmatrix}}$ is a nonzero element of $R$.

Therefore $R_R$ is essential in $W$.
\end{proof}

\begin{lemma}\label{RightE(R)} Let R=${\begin{pmatrix}

  \Z &  \Q\\

  0 & \Q

\end{pmatrix}}$. Then $E(R_R)$, the injective hull of $R$ as a right module over itself, is the right $R$-module  $W={\begin{pmatrix}

  \Q & \Q\\

  \Q & \Q

\end{pmatrix}}.$
\end{lemma}

\begin{proof} By Lemma \ref{rightRessential}, $R_R$ is essential in $W$. Also, as a ring $W$ is a semisimple Artinian ring. Then $W=Q_{max}^{r}(R)$ by \cite[Proposition 13.39]{lam}. Since $R$ is right nonsingular, $Q_{max}^{r}(M)=E(R_R)$ by Johnson's Theorem in \cite[13.36]{lam}. Therefore $E(R_R)=W$.
\end{proof}

\begin{lemma}\label{leftRessential}

  Let R=${\begin{pmatrix}

  \Z &  \Q\\

  0 & \Q

\end{pmatrix}}$. Then $_RR$ is essential in $_{R}W$=${\begin{pmatrix}

  \Q & \Q\\

  \Q & \Q

\end{pmatrix}}.$

\end{lemma}

\begin{proof}  Let ${\begin{pmatrix}

  a & b\\

  c & d

\end{pmatrix}}$ be a nonzero element of $W$. Let $a=\frac{u}{v}$ and  $c=\frac{x}{y}$, where $x,y,u,v \in \Z$ and $y \neq 0,\,\,v \neq 0.$

Case I: If $a\neq 0$ or $b \neq 0$, then ${\begin{pmatrix}

  v & 0\\

  0 & 0

\end{pmatrix}} \cdot {\begin{pmatrix}

  a & b\\

  c & d

\end{pmatrix}} ={\begin{pmatrix}

  u & vb\\

  0 & 0

\end{pmatrix}}$ is a nonzero element of $R$, because $vb \neq 0.$\\

Case II: If $c\neq 0$ or $d \neq 0$, then ${\begin{pmatrix}

  0 & y\\

  0 & 0

\end{pmatrix}} \cdot {\begin{pmatrix}

  a & b\\

  c & d

\end{pmatrix}} ={\begin{pmatrix}

  yc & yd\\

  0 & 0

\end{pmatrix}}$ is a nonzero element of $R$.

Therefore $_RR$ is essential in $_{R}W$.

\end{proof}

\begin{lemma}\label{LeftE(R)} Let $R={\begin{pmatrix}

  \Z &  \Q\\

  0 & \Q

\end{pmatrix}}$. Then $E(_RR)$, the injective hull of $R$ as a left module over itself, is the left $R$-module   $W={\begin{pmatrix}

  \Q & \Q\\

  \Q & \Q

\end{pmatrix}}.$

\end{lemma}

\begin{proof}  Similar to the proof of Lemma \ref{RightE(R)}.
\end{proof}

\begin{proposition} For the ring $ R={\begin{pmatrix}

  \Z &  \Q\\

  0 & \Q

\end{pmatrix}}$, the following statements hold.

\begin{enumerate}

\item[(1)] $R$ is right max-$QF$, but not left max-$QF$.

\item[(2)] $R$ is right almost-$QF$, but not left almost-$QF$.

\end{enumerate}

\end{proposition}

\begin{proof} $(1)$ By Lemma \ref{Rightsmall} and Lemma \ref{RightE(R)}, $R$ is a right small ring. Thus $R$ is right max-$QF$.  Consider the simple left module $S={\begin{pmatrix}

  \Z &  \Q\\

  0 & \Q

\end{pmatrix}}/  {\begin{pmatrix}

  \Z &  \Q\\

  0 & 0

\end{pmatrix}}.$  Then $S$ is singular and injective by Lemma   \ref{lem:Sisinjective}. Since $R$ is left nonsingular, the identity map $1_S : S \to S$ can not be lifted to a homomorphism $S \to R$. That is there is no homomorphism $g: S \to R$ such that $\pi g=1_S$, where $\pi : R \to S$ is the natural epimorphism. Thus the injective left $R$-module $S$ is not max-projective. Therefore $R$ is not a left max-$QF$ ring.

$(2)$ Since $R$ is right hereditary and  right noetherian, being right almost-$QF$ and right max-$QF$ coincide. Thus, $R$ is right almost-$QF$ by $(1)$. Again by $(1)$,  $R$ is not left max-$QF$, so it is not left almost-$QF$.
\end{proof}

\section{super max-QF rings}

Another fundamental question of ring theory is whether an algebraic property of a ring is transferred to quotient rings? In this section, we show that max-QF rings and almost-QF rings need not be closed under factor rings.


\begin{example} Consider the quotient ring $R=k[x,y]/(x^2,\,y^2)$ of a polynomial ring $k[x,y]$ and the ideal $I=(x^2,\,xy,\,y^2)/(x^2,y^2)$. By \cite[$\S$15, Ex. 5]{lam}, $R$ is a local QF ring. In this case, the ring $R$ is almost-$QF$, and so is max-$QF$. Now, consider the quotient ring $S=R/I \cong k[x,y]/(x^2,xy,y^2)$. Again by \cite[$\S$15, Ex. 5]{lam}, $S$ is not self-injective, this means that $S$ is not QF. On the other hand, $S$ is Artinian as a factor ring of an Artinian ring and an Artinian ring is max-QF if and only if it is QF. This means that $S$ is not a max-$QF$ ring, and so is not an almost-$QF$ ring.

\end{example}

Recall that a ring $R$ is called super $QF$ if every quotient ring of $R$ is $QF$. Note that the example above shows that factor rings of almost-QF (resp. max-$QF$) rings need not be almost-$QF$ (resp. max-$QF$). Thus, it is natural to consider the rings whose every quotient ring is almost-QF (resp. max-$QF$).

\begin{definition}  $R$ is said to be right super almost-QF (respectively,  super max-QF) if every quotient ring of $R$ is right almost-QF (respectively, right max-QF).

\end{definition}

Since being $QF$ is left-right symmetric, a ring $R$ is left super-$QF$ if and only if $R$ is right super $QF$. Every super $QF$ ring is left-right super almost-$QF$, and every right super almost-$QF$ ring is right super max-$QF$. But a super almost-$QF$ ring need not be super $QF$ by the following lemma.

\begin{lemma} Let $R$ be a $PID$ which is not a field. Then $R$ is super almost-$QF$, but not $QF$.
\end{lemma}

\begin{proof} Every domain is small and so almost-$QF$. It is well known that every proper factor ring of a PID is $QF$. Thus every $PID$ is super almost-$QF$. As $R$ is a domain which is not a field, $R$ is not $QF$.
\end{proof}

Over a right perfect ring, the notions of projectivity, $R$-projectivity and max-projectivity coincide (see, \cite[Corollary 4]{maxproj}). Also, right Artinian rings are closed under factor rings and both left and right perfect. Hence we have the following:

\begin{proposition} Let $R$ be right Artinian ring. The following are equivalent.

\begin{enumerate}

 \item[(1)] $R$ is right super almost-$QF$.

\item[(2)] $R$ is right super max-$QF$.

\item[(3)] $R$ is  super $QF$.

\item[(4)] $R$ is left super almost-$QF$.

\item[(5)] $R$ is left super max-$QF$.

\end{enumerate}
\end{proposition}

\begin{proposition} If $R$ is a commutative ring, then $R/P$ is a max-$QF$ ring for each prime ideal $P$.

\end{proposition}

\begin{proof} Note that, an ideal of $P$ of a commutative ring $R$ is prime if and only if the factor ring $R/P$ is an integral domain. Every integral domain  is a small ring, and so is max-$QF$. Thus the proof follows.
\end{proof}

\section*{Acknowledgement}

The authors are supported by T\"{U}B\.{I}TAK-The Scientific and Technological Research Council of T\"{u}rkiye-under the project with reference 122F158.

%
%
%


\begin{thebibliography}{99}

\renewcommand{\baselinestretch}{1}

\small\normalsize

\bibitem{testingforprojectivity} H. Alhilali, Y. Ibrahim, G. Puninski, and M. Yousif, When R is a testing module for projectivity? {\it J. Algebra} {\bf 484} (2017), 198-206.

\bibitem{maxproj} Y. Alag\"{o}z, E. B\"{u}y\"{u}ka\c{s}{\i}k, Max-projective modules. {\it J. Algebra Appl.} {\bf 20}(6) (2021), 2150095.

\bibitem{nonsingular} Y. Alag\"{o}z, S. Benli and E. B\"{u}y\"{u}ka\c{s}{\i}k, Rings whose nonsingular right modules are $R$-projective, {\it Comment.Math. Univ. Carolin.}, {\bf 62}(4) (2021), 393-407.

\bibitem{A-perfect} A. Amini, M. Ershad and H. Sharif, Rings over which flat covers of finitely generated modules are projective, {\it Comm. Algebra} {\bf 36}(8) (2008), 2862-2871.



%
%

\bibitem{B-perfect}
E. B\"{u}y\"{u}ka\c{s}\i k, Rings over which flat covers of simple modules are projective, {\it J. Algebra Appl.} {\bf 11}(3) (2012), 1250046, 7.



\bibitem{radsupp}
E. B\"{u}y\"{u}ka\c{s}\i k, E. Mermut, and S. \"{O}zdemir, Rad-supplemented modules, {\it Rend. Semin. Mat. Univ. Padova} {\bf 124} (2010), 157-177.

\bibitem{primehereditarynoether} H. Q. Dinh, C. J. Holston, and D. V. Huynh, Quasi-projective modules over prime hereditary Noetherian V-rings are projective or injective, {\it J. Algebra} {\bf 360} (2012), 87-91.

\bibitem{FaithQF} C. Faith, {\it Algebra. II}, Springer-Verlag, Berlin-New York, 1976. Ring theory, Grundlehren der Mathematischen Wissenschaften, No. 191.

%
%
%
%


\bibitem{goodearl} K. R. Goodearl, {\it Ring theory}. Marcel Dekker Inc., New York, 1976.





\bibitem{lamfc} T. Y. Lam, {\it A first course in noncommutative rings}, Second, Graduate Texts in Mathematics, vol. 131, Springer-Verlag, New York, 2001.


\bibitem{lam} T. Y. Lam, {\it Lectures on modules and rings}, Springer-Verlag, New York, 1999.




\bibitem{lomp2} C. Lomp, On the splitting of the dual Goldie torsion theory. {\it Algebra and its Applications} (Athens, OH), 377-386, Contemporary Mathematics 259, American Mathematical Society, Providence, RI, (2000) 377-386.


\bibitem{simples-pure-inj} L. Mao and N. Ding, Cotorsion modules and relative pure-injectivity, {\it J. Austral. Math. Soc.} 81 (2006) 225–243.

\bibitem{matlis} E. Matlis, Injective modules over Noetherian rings, {\it Pacific J. Math.} {\bf 8}(3) (1958), 511-528.


\bibitem{megibben} C. Megibben, Absolutely pure modules,  {\it Proc.  American Math. Soc.} {\bf 26-4}, (1970), 561-566.

\bibitem{Nakayama} T. Nakayama, On Frobeniusean algebras I, {\it Ann. of Math.} 40, (1939) 611–633.

\bibitem{quasi-frobenious} W. K. Nicholson and M. F. Yousif: Quasi-Frobenius Rings. Cambridge Tracts in Math., 158, Cambridge University Press, (2003).

\bibitem{ozcan} A. \c{C}. \"Ozcan and A. Harmanc{\i}, Characterization of some rings by functor $Z^{*}()$, {\it Turkish J. Math.}, {\bf 21}(3) (1997), 325-331.



\bibitem{ramamurthi} V. S. Ramamurthi,  The smallest left exact radical containing the Jacobson radical. {\it Ann. Soc. Sci. Bruxelles Ser. I.} {\bf 96}(4) (1982), 201-206.


%
%

\bibitem{sandomierski-closed} F. L. Sandomierski, Nonsingular rings, {\it Proc. Amer. Math. Soc.} 19 (1968), 225–230.




\bibitem{trifilajfaithproblem} J. Trlifaj, Faith's problem on R-projectivity is undecidable, {\it Proc. Amer. Math. Soc.} {\bf 147} (2019), no. 2, 497-504.


\bibitem{dualbeartrilifaj} J. Trlifaj, Dual Baer Criterion for non-perfect rings, {\it Forum Math.} {\bf 32}(3) (2020), 663-672.

\end{thebibliography}
\end{document}